\documentclass[a4,12pt]{amsart}       

\usepackage[utf8]{inputenc}
\usepackage[T1]{fontenc}
\usepackage{yfonts}
\usepackage[english]{babel}
\usepackage[normalem]{ulem}
\usepackage{lipsum}
\usepackage{amsmath}
\usepackage{amsthm}
\usepackage{amssymb}

\usepackage[shortlabels]{enumitem}
\usepackage{graphicx}
\usepackage{mathtools}
\usepackage{hyperref}
\usepackage{amsfonts}
\usepackage{latexsym}
\usepackage{amscd}

\usepackage[all]{xy}
\usepackage[dvipsnames]{xcolor}
\usepackage{stmaryrd}
\DeclareMathOperator{\Imm}{Im}

\DeclareMathOperator{\Ker}{Ker}
\DeclareMathOperator{\End}{End}
\DeclareMathOperator{\Hom}{Hom}

\DeclareMathOperator{\Spec}{Spec}

\def\n{\overline n}
\def\c{\overline c}
\def\dualita#1#2{\mathrel{
                 \mathop{\vcenter{
                 \offinterlineskip
                 \hbox to 1.2truecm{$\mapsto$}
                 \hbox to 1.2truecm{$\mapsfrom$}}}%
                 }}
\newtheorem{theorem}{Theorem}[section]
\newtheorem{lemma}[theorem]{Lemma}
\newtheorem{proposition}[theorem]{Proposition}
\newtheorem{corollary}[theorem]{Corollary}

\newtheorem{remark}[theorem]{Remark}

\begin{document}
\title[Indecomposable Injectives]{Indecomposable Injectives  over the Jacobson Algebra}


\author{Riccardo Colpi}
\address{Dipartimento di Matematica ``Tullio Levi-Civita'', Universit\`{a} degli Studi di Padova, I-35121, Padova, Italy}
\email{riccardo.colpi@unipd.it}

\author{Francesca Mantese}
\address{Dipartimento di Informatica, Universit\`{a} degli Studi di Verona, I-37134 Verona, Italy}
\email{francesca.mantese@univr.it}

\author{Alberto Tonolo}
\address{Dipartimento di Matematica ``Tullio Levi-Civita'', Universit\`{a} degli Studi di Padova, I-35121, Padova, Italy - https://orcid.org/0000-0002-9844-3998}
\email{alberto.tonolo@unipd.it}

\thanks{
The authors are supported by Project funded by the EuropeanUnion – NextGenerationEU under the National Recovery and Resilience Plan (NRRP), Mission 4 Component 2 Investment 1.1 - Call PRIN 2022 No. 104 of February 2, 2022 of Italian Ministry of University and Research; Project 2022S97PMY (subject area: PE - Physical Sciences and Engineering) "Structures for Quivers, Algebras and Representations (SQUARE)". They are moreover members of INDAM - GNSAGA}
%
%
%
%
%

%


%
\maketitle

%

\begin{abstract}
Let $K$ be any field. In this paper we give a complete list of the indecomposable left injective module over the Jacobson algebra $K\langle X,Y\mid XY=1\rangle$, i.e., the free associative $K$-algebra on two (noncommuting) generators, modulo the single relation $XY=1$. This is the natural continuation of the paper of the second two authors with Gene Abrams on the charaterization of the injective envelope of the simple modules over $K\langle X,Y\mid XY=1\rangle$.


\normalsize

\end{abstract}
\section*{Introduction}

This work builds upon the study of injective modules over the Jacobson algebra that the second two authors initiated in a joint paper \cite{AMT21} with  Gene Abrams. This algebra was explicitly studied by Jacobson in the late 1940s in \cite{J40}. In \cite{AMT21}, the characterization of the injective envelopes of the left simple modules was obtained. The aim of this paper is to extend our investigation to include a comprehensive list of all the left indecomposable injectives.

Given an arbitrary field $K$, the Jacobson $K$-algebra $K\langle X,Y\mid XY=1\rangle$ is the free associative $K$-algebra on two (noncommuting) generators, modulo the single relation $XY=1$. The structure of the ring $K\langle X,Y\mid XY=1\rangle$ is quite interesting. Over the past seventy years since Jacobson's work, many mathematicians have extensively examined numerous properties of $K\langle X,Y\mid XY=1\rangle$ in both the ring-theoretic and module-theoretic contexts. These investigations have been conducted in various works, such as Cohn \cite{C66} (1966); Bergman \cite{B74} (1974); Gerritzen \cite{G00} (2000); Bavula \cite{Bav10} (2010); Ara and Rangaswamy
\cite{AR14} (2014); Iovanov and Sistko \cite{IS17} (2017); Lu, Wang and Wang \cite{LWW19} (2019). 

The Jacobson algebra $K\langle X,Y\mid XY=1\rangle$ is isomorphic to the Leavitt path algebra $R:=L_K(\mathcal T)$ associated to the directed graph
$\mathcal T:=\xymatrix{v\ar@(ul,dl)_c
\ar[r]^d&w}$. We give an explicit isomorphism below. For those wishing to delve deeper into the understanding of Leavitt path algebras, we recommend referring to \cite[Chapter 1]{AAM}. However, it is still possible to read this work without an in-depth knowledge of Leavitt path algebras, as it relies on the description of the $K$-algebra $R$ in terms of generators and relations. The description of the Jacobson's algebra as the Leavitt path algebra $R$ associated to the graph $\mathcal T$ allows, in our opinion, a better understanding of its algebraic structure.

In the paper \cite{AMT21}, the injective envelope of each simple left $R$-module (see \cite{AMT24} for a generalization to the Leavitt path algebras associated to graphs with disjoint cycles) was constructed. The 
injective envelopes of the simple left $R$-modules are clearly indecomposable, but not every indecomposable injective possesses a simple submodule. Therefore the list of indecomposable injective left $R$-modules is not necessarily complete. We will prove that only one was missing.

\section{Prerequisites and notation}

Hereinafter, by the unadorned term ``module'', we will always mean ``left module'' whenever working with non-commutative rings.
The Leavitt path algebra $L_K(\mathcal T)$ associated to the directed graph $\mathcal T:=\xymatrix{{\bullet}^v\ar@(ul,dl)[]|c\ar[r]^d&{\bullet}^w}$ is the quotient of the $K$-quiver path algebra  associated to the \emph{extended graph} $\widehat{\mathcal T}$ of $\mathcal T$, pictured as
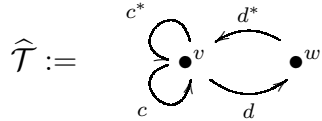
\[\widehat{\mathcal T}:=\xymatrix{  \hskip-.3in & {\bullet}^v \ar@(l,d)_c \ar@(u,l)_{c^*} 
 \ar@/^-1.0pc/[r]_d 
&   \ar@/^-1.0pc/[l]_{d^*}
  {\bullet}^w}\]
modulo the relations
  \[c^*c=v,\  d^*d=w,\  c^*d=d^*c=0,\  cc^*+dd^*=v.\]
  A $K$-base of $L_K(\mathcal T)$  is
  \[\{v,w,d,d^*,c^i,c^id, c^i(c^*)^j, (c^*)^j, d^*(c^*)^j: i,j\geq 1\}.\]
  A particular isomorphism between the Jacobson algebra $K\langle X,Y\mid XY=1\rangle$ and $L_K(\mathcal T)$  is realised sending $X\mapsto c^*+d^*$ and $Y\mapsto c+d$; the inverse isomorphism is realised sending \label{isoJac} $v\mapsto YX$, $w\mapsto 1-YX$, $c\mapsto Y^2X$, $c^*\mapsto YX^2$, $d\mapsto Y-Y^2X$, $d^*\mapsto X-YX^2$. In the sequel  we will shortly denote by $R$ the Leavitt path algebra $L_K(\mathcal T)$:
  \[R:=L_K(\mathcal T).\]
  
  \medskip

  \noindent The socle $J:=\text{Soc}(R)$ of $R$ is the two sided ideal  $\langle w\rangle$ generated by the idempotent $w$ (\cite[Theorem 2.6.14]{AAM}). It coincides with the following infinite direct sum of  left $R$-modules
\[J=Rw\oplus\bigoplus_{i\geq 0}Rd^*(c^*)^i.\]
\begin{remark}\label{iso:rightprod}
For any $i\geq 0$ the direct summand $Rd^*(c^*)^i$ of $J$ is isomorphic to the direct summand $Rw$ through the right product by $d^*(c^*)^i$:
\[\xymatrix{Rw\ar[rr]^-{-\times d^*(c^*)^i}&&Rd^*(c^*)^i,}\quad rw\mapsto rwd^*(c^*)^i=rd^*(c^*)^i.
\]
\end{remark}
Since $Rw$ is a simple  $R$-module (\cite[Corollary~4.6]{AR14}), the one exhibited above is indeed the decomposition that describes the semisimple nature of the socle $J$. 
\medskip

\begin{remark}\label{rem:leftideals} {As we proved in \cite[Proposition~5.3]{AMT21},  any left ideal of $R$ is either a direct summand of $J$ or a direct summand of a free ideal $Rp(c)$, where $p(x)=1+p_1x+\cdots p_mx^m\in K[x]$ is any polynomial with constant term equal to $1$ and $p(c)=1_R+p_1c+\cdots+p_mc^m$ is the corresponding element in $R$.}
\end{remark}
\medskip

\noindent The $K$-algebra  $R/J$ is isomorphic to the $K$-algebra of Laurent polynomials (\cite[Example 4.5]{R15}). Indeed, setting $\c=c+J$, we have $c^*+J=\c^{-1}$ and $R/J=K[\c,\c^{-1}]$.
%
%

\begin{remark}\label{PIDLAURENT} The ring $K[\c,\c^{-1}]$ of Laurent polynomials {in $\c$} is a principal ideal domain. Its ideals are of the form $K[\c,\c^{-1}]p(\c)$ where $p(\c)$ is the evaluation in $\c$ of any polynomial $p(x)$ in $K[x]$. In particular the spectrum of prime ideals consists of the 0 ideal and all the maximal ideals, i.e., the ideals of the form $K[\c,\c^{-1}]f(\c)$ where $f(x)$ is an irreducible polynomial in $K[x]$ with $f(0)=1$. 
By Matlis' Theorem (see \cite[3.62]{Lam1}) the set \[\{E(K[\c,\c^{-1}]/P):P\in\Spec K[\c,\c^{-1}]\}\] is a complete list of indecomposable injective $K[\c,\c^{-1}]$-modules. If  $P=0$, then $E(K[\c,\c^{-1}])$ is the quotient field $K(\c)$. Any non-zero prime ideal in $K[\c,\c^{-1}]$ has the form $P=\langle f(\c)\rangle$ where $f(x)$ is an irreducible polynomial in $K[x]$ with $f(0)=1$. The injective envelope 
$E(K[\c,\c^{-1}]/P)$ {coincides with} the $f(\c)$-torsion submodule of $K(\c)/K[\c,\c^{-1}]$, i.e., $K(\c)/K[\c,\c^{-1}]S_{f(\c)}^{-1}$ where $S_{f(\c)}$ denotes the multiplicative set $K[\c,\c^{-1}]\setminus \langle f(\c)\rangle$ (see \cite[3.63]{Lam1}). The module $E(K[\c,\c^{-1}]/\langle f(\c)\rangle)$ has the filtration
\[(0)<K[\c,\c^{-1}]f(\c)^{-1}/K[\c,\c^{-1}]< K[\c,\c^{-1}]f(\c)^{-2}/K[\c,\c^{-1}]<...\]
with each filtration factor isomorphic to $K[\c,\c^{-1}]/K[\c,\c^{-1}]f(\c)$ and hence it is isomorphic to
\[\bigcup_{i\geq 1}K[\c,\c^{-1}]/K[\c,\c^{-1}] f(\c)^i=\varinjlim_i K[\c,\c^{-1}]/K[\c,\c^{-1}]f(\c)^i.\]
\end{remark}

\medskip

\noindent In the following lemma we stress the connection between injective modules over $R$ and over the quotient ring $R/J=K[\c,\c^{-1}]$. 

\begin{lemma}\label{injsuRsuRJ}   If $E$ is an injective  $R$-module and $JE=0$, then $E$ is an injective  $R/J$-module. Conversely,   If $E'$ is an injective $R/J$-module, then it is an injective $R$-module.
\end{lemma}
\begin{proof}  By \cite[Corollary 3.6B]{Lam1}, if $E$ is an injective  $R$-module, then $\Hom_R(R/J, E)$ is an injective  $R/J$-module; in particular if $E$ is injective and $JE=0$, then $E\cong \Hom_R(R,E)\cong \Hom_R(R/J,E)$ is an injective  $R/J$-module. This implication is true for any associative ring and any two-sided ideal.\\
Conversely, assume $E'$ to be  an injective $R/J$-module. Then $E'$ is also a $R$-module with $J\cdot E'=0$.
Reminding the structure of the left ideals  of $R$ recalled in Remark~\ref{rem:leftideals}, let us check that any  morphism of  $R$-modules of type   $\psi:J\to E'$ or of type $\varphi_p:Rp(c)\to E'$, for  $p(x)\in K[x]$,  lifts to  a morphism from $R$ to $E'$. 
Since $JE'=0$ we get $\psi(w)=w\psi(w)=0$ and  $\psi(d^*(c^*)^i)=w\psi(d^*(c^*)^i)=0$: hence $\psi=0$. Then $\psi$ trivially lifts to the zero map  from $R$  to  $E'$. Let now $p(x)\in K[x]$  be any polynomial with $p(0)=1$.  Since 
\begin{align*}
&Rw=Rwp(c)\leq Jp(c),\\
&Rd^*=Rd^*p(c)\leq Jp(c),\\
&Rd^*c^*\leq Rd^*+Rd^*c^*p(c)\leq Jp(c), ...,\text{  and}\\
&Rd^*(c^*)^i \leq Rd^*+Rd^*c^*+\cdots+Rd^*(c^*)^{i-1}+Rd^*(c^*)^ip(c)\leq Jp(c)
\end{align*}
for each $i\geq 0$,
  we get  $Jp(c)=J$. As $JE'=0$, it follows that any morphism  $\varphi_p:Rp(c)\to E'$ vanishes on $J$ and hence it defines a morphism of $R/J$-modules $\varphi'_p:Rp(c)/J\to E'$ which lifts to a morphism of $R/J$-modules $\overline{\varphi'_p}:R/J\to E'$. Setting $\overline{\varphi_p}(1_R):=\overline{\varphi'_p}(1_R+J)$ we define a morphism of  $R$-modules extending $\varphi_p$.
\end{proof}

\begin{remark}\label{rem:K(C)}
As observed in Remark~\ref{PIDLAURENT}, 
\[K(\overline c)\quad\text{and}\quad \varinjlim_i K[\overline c, {\overline c}^{-1}]/K[\overline c, {\overline c}^{-1}]f(\overline c)\]
where $f(x)$ is any irreducible polynomial in $K[x]$ with $f(0)=1$, is a complete list of indecomposable injective $K[\overline c, {\overline c}^{-1}]$-modules. By Lemma~\ref{injsuRsuRJ} these are also injective left $R$-modules. Moreover, since the lattices of left $R$- and $R/J$-submodules is the same for any $R/J$-module, they are also indecomposable left $R$-modules. In the sequel we will prove that these together with the injective envelope $E(Rw)$ of the simple $Rw$ form a complete list up to isomorphisms of the indecomposable injective left $R$-modules.
\end{remark}
\medskip

\noindent By \cite[Corollary 4.6]{AR14} a complete list of  the non-isomorphic simple left $R$-modules is given by the left ideal $Rw$ and the  $R$-modules $R/Rf(c)$, where $f(x)$ varies among  irreducible polynomials in $K[x]$ with fixed non-zero constant term, which we assume to be equal to 1.

\medskip

The injective envelope of the simple module $R/Rf(c)$  is the Pr\"ufer  module \cite[\text{Corollary 6.3}]{AMT21}
\[U^f=\varinjlim_n R/Rf^n(c).\]

\noindent Since $Jf(c)=J$, the $R$-module $R/Rf(c)$ is also a $R/J=K[\c,\c^{-1}]$-module and 
\begin{align*}
R/Rf(c)\cong (R/J)/(Rf(c)/J)&=(R/J)/(R/J)f(c)\\
&=K[\c,\c^{-1}]/K[\c,\c^{-1}] f(\c).\end{align*}

By the same argument, also  $U^f$ is a \ $R/J=K[\c,\c^{-1}]$-module and we get as $K[\c,\c^{-1}]$-modules
\begin{align*}
U^f&=\varinjlim_n R/Rf^n(c)\cong \varinjlim_n (R/J)/(R/J)f^n(\c)\\
&=\varinjlim_n K[\c,\c^{-1}]/ K[\c, \c^{-1}]f^n(\c) \end{align*}
which is the injective indecomposable $K[\c,\c^{-1}]$-module associated to $f(x)$ 
described in Remark~\ref{PIDLAURENT}. 

\medskip
The left ideal $Rw$ is clearly a $K$-vector space containing the subspaces $Kw$, $Kd$, $Kcd$, $Kc^2d$, \dots. We denote by $[c](Kd)$ the sum of the subspaces $Kc^id$, $i\geq 0$, to emphasise the fact that this is the subspace of polynomials in $c$ with right coefficients in $Kd$. We denote by 
\[[[c]](Kd):=\{\sum_{i\geq 0}k_ic^id\mid k_i\in K\}\]
the $K$-vector space of formal series in $c$ with right coefficients in $Kd$.

\noindent The injective envelope of $Rw$ is equal to the  $R$-module of ``formal series'' \cite[Corollary 6.12]{AMT21}
\[E(Rw)=\{k_{-1}w+\sum_{i\geq 0}k_ic^id\mid k_i\in K\}=:Kw+[[c]](Kd).\] 

\medskip

\noindent Since $E(Rw)=Kw+[[c]](Kd)$ and the Pr\"ufer  $R$-modules $U^f=\varinjlim_n R/Rf^n(c)$ are  injective envelopes of simple modules, they are indecomposable. Are there other indecomposable injective $R$-modules?

\noindent Any Leavitt path algebra, so in particular $R$, is (left and right) hereditary \cite[Theorem~3.7]{AG12}, and thus the quotient of any injective left module is again injective.  

\begin{lemma}\label{lemma:imaggini}
Let $S$ be any left hereditary ring, e.g., $S=R$. Every indecomposable injective left $S$-module is an homomorphic image of the injective envelope $E({}_SS)$ of the left regular module.
\end{lemma}
\begin{proof}
Let $E$ be an  indecomposable injective left $S$-module and let $\varphi:S\to E$ be a non-trivial morphism. Then $\varphi$ extends to a morphism $\tilde{\varphi}: E(S)\to E$ and $\Imm\tilde{\varphi}$ is a non-zero injective submodule of $E$. Since $E$ is indecomposable we get $\Imm\tilde{\varphi}=E$.
\end{proof}

\medskip

\noindent In the next section we study carefully  the structure of the $K$-algebra $R$ in order to subsequently investigate the injective envelope of the left regular $R$-module and its homomorphic images. This will permit us to obtain our aim.




\section{The injective envelope of the regular module}

Since $v$ and $w$ are orthogonal idempotents and $1=v+w=cc^*+dd^*+w$, we get that 
\[R=Rv\oplus Rw=Rcc^*\oplus Rdd^*\oplus Rw=Rc^*\oplus Rd^*\oplus Rw.\]
As observed in the discussion prior to Remark~\ref{iso:rightprod}, the left ideals $Rd^*$ and $Rw$ are isomorphic simple left $R$-modules (see \cite[Corollary 4.6]{AR14}).  Notice that, for the element $cc^*$, we have the following decomposition as sum of orthogonal idempotents
\begin{align*}
cc^*&=c^2(c^*)^2+\big(c(1-cc^*)c^*\big)=c^2(c^*)^2+(cdd^*c^*),\\ c^2(c^*)^2&=c^3(c^*)^3+\big(c^2dd^*(c^*)^2\big),\\
 \dots\end{align*}
which leads to the corresponding decompositions 
\[Rcc^*=Rc^*=R(c^*)^2\oplus Rd^*c^*=\big(R(c^*)^3\oplus Rd^*(c^*)^2\big)\oplus Rd^*c^*=
\cdots\]
Thus, for any $n\geq 2$ we have 
\[Rc^*=R(c^*)^n\oplus\bigoplus_{i=1}^{n-1}Rd^*(c^*)^i.
\]
All  the  direct summands $Rd^*(c^*)^i$ are simple since they are isomorphic to  $Rw$, as we remarked in the previous section. The direct summand $R(c^*)^n$ is isomorphic to $Rc^*$ through the right multiplication by $(c^*)^{n-1}$ for each $n\geq 2$.

%

Passing to the injective envelopes,  we  have 
\[E(R)=E(Rc^*)\oplus E(Rd^*)\oplus E(Rw)\cong E(Rc^*)\oplus E(Rw)\oplus E(Rw).\]
 While the module $E(Rw)$ is an indecomposable injective, the module $E(Rc^*)$ is clearly not indecomposable since for any $n\geq 2$ we have 
\[E(Rc^*)= E(R(c^*)^n)\oplus \bigoplus_{i=1}^{n-1}E(Rd^*(c^*)^i)
.
\] 
\begin{remark}\label{rem:directsum}
As a result, $E(Rc^*)$ contains the infinite direct sum $\bigoplus_{j\geq 1}E(Rd^*(c^*)^j)$ of indecomposable injectives, all isomorphic to $E(Rw)$.  The isomorphism between  $Rw$ and $Rd^*(c^*)^j$, given by the right multiplication by $d^*(c^*)^j$, extends to an isomorphism between $E(Rw)$ and $E(Rd^*(c^*)^j)$. Thus, since $E(Rw)=Kw+[[c]](Kd)$, we have 
\[E(Rd^*(c^*)^j)=Kd^*(c^*)^j+[[c]](Kdd^*(c^*)^j)\] and  hence
\begin{align*}\bigoplus_{j\geq 1}E(Rd^*(c^*)^j)&=\big(Kd^*c^*+[[c]](Kdd^*c^*)\big)[c^*]\\
&=\big(Kd^*c^*+[[c]](K(v-cc^*)c^*)\big)[c^*].
\end{align*}
It is easy to check that the latter is not an injective module: indeed the embeddings
\[Rd^*(c^*)^i\hookrightarrow E(Rd^*(c^*)^i)\hookrightarrow \bigoplus_{j\geq 1}E(Rd^*(c^*)^j), \quad i\geq 1,\]
define a morphism $R\geq \bigoplus_{j\geq 1}Rd^*(c^*)^i\to \bigoplus_{j\geq 1} E(Rd^*(c^*)^j)$ which does not extend to $R$.
\end{remark}
 

\medskip

\noindent Let us study the  $R$-module $E(Rc^*)$. A $K$-base for $Rc^*$ is given by the elements $c^i(c^*)^{j}$ and $d^* (c^*)^{j}$, where $i\geq 0, j\geq 1$.
The set
\[d^*c^*K[c^*]+[[c]](c^*K[c^*])=\{d^*c^*q_{-1}(c^*)+\sum_{i\geq 0}c^ic^*q_i(c^*)\mid q_i(x)\in K[x]\}\]
is clearly a $K$-vector space; in the next lemma we show it has also a natural structure of left $R$-module.
\begin{proposition}\label{prop:modulo}
The $K$-vector space $d^*c^*K[c^*]+[[c]](c^*K[c^*])$ is a left $R$-module.
\end{proposition}
\begin{proof}
Let us denote by $\Theta$ the $K$-vector space $d^*c^*K[c^*]+[[c]](c^*K[c^*])$.
The elements $v$, $w$, $c$, $c^*$, $d$, and $d^*$ form a system of generators for the $K$-algebra $R$. Let us denote by $P_v$, $P_w$, $P_c$, $P_{c^*}$, $P_d$, and $P_{d^*}$ the following maps of the $K$-vector space $\Theta$ in itself:
\begin{align*}
\label{}
 P_v\left(d^*c^*q_{-1}(c^*)+\sum_{i\geq 0}c^ic^*q_i(c^*)\right)   &=\sum_{i\geq 0}c^ic^*q_i(c^*))   \\
 P_w\left(d^*c^*q_{-1}(c^*)+\sum_{i\geq 0}c^ic^*q_i(c^*)\right)   &=d^*c^*q_{-1}(c^*)\\
 P_c\left(d^*c^*q_{-1}(c^*)+\sum_{i\geq 0}c^ic^*q_i(c^*)\right)   &=\sum_{i\geq 0}c^{i+1}c^*q_i(c^*))   \\
 P_{c^*}\left(d^*c^*q_{-1}(c^*)+\sum_{i\geq 0}c^ic^*q_i(c^*)\right)   &=c^*(c^*q_0(c^*))+\sum_{i\geq 1}c^{i-1}c^*q_i(c^*))   \\
 P_d\left(d^*c^*q_{-1}(c^*)+\sum_{i\geq 0}c^ic^*q_i(c^*)\right)   &=dd^*c^*q_{-1}(c^*)=(v-cc^*)c^*q_{-1}(c^*)\\
 &=c^*q_{-1}(c^*)-c(c^* q_{-1}(c^*)) \\
 P_{d^*}\left(d^*c^*q_{-1}(c^*)+\sum_{i\geq 0}c^ic^*q_i(c^*)\right)   &=d^*c^*q_{0}(c^*).
 \end{align*}
 It is easy to see that the above maps are $K$-linear maps. Let us see that setting $v\mapsto P_v$, $w\mapsto P_w$, $c\mapsto P_c$, $c^*\mapsto P_{c^*}$, $d\mapsto P_d$, and $d^*\mapsto P_{d^*}$, we define a ring homomorphism between $R$ and the ring $\End_K \Theta$ of  the left endomorphisms \cite[Ch. II, \S 2]{AF92} of the $K$-vector space $\Theta$. As a consequence, $\Theta$ results to be a left $R$-module. 
 The ring $R$ is the free associative $K$-algebra generated by  $v$, $w$, $c$, $c^*$, $d$, $d^*$ subject to the following relations:
 \[vw=0=wv,\ v^2=v,\ w^2=w,\ vc=cv=c, \ vc^*=c^*v=c^*,\]
 \[vd=d=dw, \ wd^*=d^*=d^*v, \ \text{and}
 \]
 \[c^*c=v,\  d^*d=w,\  c^*d=d^*c=0,\  cc^*+dd^*=v.
 \]
 To prove that we have a well defined ring homomorphism, it is necessary to verify that the endomorphisms $P_v$, $P_w$, $P_c$, $P_{c^*}$, $P_d$, and $P_{d^*}$ satisfy the corresponding relations and that $1_R\mapsto 1_{\End_K\Theta}$.

We verify explicitly that $P_v=P_{c^*}\circ P_c$ and $P_v\circ P_w=0$; analogously the other relations can be proven.
\begin{align*}
(P_{c^*}\circ P_c)\left(d^*c^*q_{-1}(c^*)+\sum_{i\geq 0}c^ic^*q_i(c^*)\right)&=P_{c^*}\left(\sum_{i\geq 0}c^{i+1}c^*q_i(c^*)\right)\\
&=P_{c^*}\left(\sum_{i\geq 1}c^{i}c^*q'_{i}(c^*)\right),
\end{align*}
where $q'_i(c^*)=q_{i-1}(c^*)$,
\begin{align*}
{\phantom{(P_{c^*}\circ P_c)\left(d^*c^*q_{-1}(c^*)+\right)}}&=\sum_{i\geq 1}c^{i-1}c^*q'_{i}(c^*)\\
&=\sum_{i\geq 1}c^{i-1}c^*q_{i-1}(c^*)\\
&=\sum_{i\geq 0}c^{i}c^*q_{i}(c^*)\\
&=P_v\left(d^*c^*q_{-1}(c^*)+\sum_{i\geq 0}c^ic^*q_i(c^*)\right);
\end{align*}

\begin{align*}
(P_{v}\circ P_w)\left(d^*c^*q_{-1}(c^*)+\sum_{i\geq 0}c^ic^*q_i(c^*)\right)&=P_{v}\left(d^*c^*q_{-1}(c^*)\right)\\
&=0.
\end{align*}

Finally, let us check that $1_R=v+w\mapsto 1_{\End_K\Theta}$:
 \[(P_v+P_w)\left(d^*c^*q_{-1}(c^*)+\sum_{i\geq 0}c^ic^*q_i(c^*)\right)=\sum_{i\geq 0}c^ic^*q_i(c^*))+d^*c^*q_{-1}(c^*).\]
\end{proof}

\begin{remark}
Observe that, since $dd^*=v-cc^*$, $([[c]]Kdd^*c^*)[c^*]$ is a $K$-vector subspace of  the space $([[c]]Kc^*)[c^*]$ of polynomials in $c^*$ with coefficients in the $K$-vector space $[[c]]c^*K$ of formal series in $c$ multiplied by $c^*$. Moreover, the space  $([[c]]Kc^*)[c^*]$ is strictly contained in the vector space $[[c]](c^*K[c^*])$ of formal series in $c$ with coefficients in the $K$-vector space $c^*K[c^*]$ of polynomials in  $c^*$ multiplied by $c^*$.
\end{remark}

We can now prove that
\begin{theorem}
The  $R$-module $d^*c^*K[c^*]+[[c]](c^*K[c^*])$ of Proposition~\ref{prop:modulo} is the injective envelope $E(Rc^*)$ of $Rc^*$.
\end{theorem}
\begin{proof}
First let us show that $Rc^*$ is essential in $d^*c^*K[c^*]+[[c]](c^*K[c^*])$. If $\alpha:=d^*c^*q_{-1}(c^*)+\sum_{i\geq 0}vc^ic^*q_i(c^*)\not=0$ is any non-zero element in  $d^*c^*K[c^*]+[[c]](c^*K[c^*])$, let $\overline i:=\min\{i\geq -1: q_i(x)\not=0\}$. If $\overline i=-1$, then
$w\alpha=d^*c^*q_{-1}(c^*)$ is a non-zero element in $Rc^*$. If $\overline i\geq 0$, then
\[d^*(c^*)^{\overline i}\alpha=d^*(c^*)^{\overline i}\left(\sum_{i\geq \overline i}vc^ic^*q_i(c^*)\right)=d^*c^*q_{\overline i}(c^*)\]
is a non-zero element in $Rc^*$.

Second let us prove that  $d^*c^*K[c^*]+[[c]](c^*K[c^*])$ is injective. Denote by $\mathcal P$ the set $\{p(x)\in K[x] \ | \ p(0)=1\}$. As a consequence of Remark~\ref{rem:leftideals}, in order to apply the Baer criterion to determine the injectivity of $d^*c^*K[c^*]+[[c]](c^*K[c^*])$, we just have to  check the injectivity, (1) with respect to $J$, and (2) with respect to left ideals of the form $Rp(c)$, for $p(x)\in \mathcal P$. For (2), let $\varphi:Rp(c)\to d^*c^*K[c^*]+[[c]](c^*K[c^*])$ be a non-zero homomorphism of  $R$-modules with $p(x)=p_0+p_1x+\cdots p_nx^n\in\mathcal P$. Since  in our assumption $p_0=1\not=0$, the polynomial $p(x)$ is invertible in $K[[x]]$. Denote by $B(x)=\sum_{i\geq 0}b_ix^i$ the formal series such that
$B(x)p(x)=p(x)B(x)=1$. Then
\[p_0b_0=1, \text{ and }\sum_{j=0}^N p_jb_{N-j}=0\text{ for all }N\geq 1.\]
If $\varphi(p(c))=d^*c^*q_{-1}(c^*)+\sum_{i\geq 0}c^ic^*q_i(c^*)$, consider the element
\[\beta:=b_0d^*c^*q_{-1}(c^*)+\sum_{i\geq 0}c^i\left(\sum_{j=0}^ib_jc^*q_{i-j}(c^*)\right) \in d^*c^*K[c^*]+[[c]](c^*K[c^*]).\]
Setting $\overline\varphi(1)=\beta$, we define an homomorphism 
\[\overline\varphi:R\to d^*c^*K[c^*]+[[c]](c^*K[c^*])\] which extends $\varphi$. Indeed,
\begin{align*}
\overline\varphi(p(c))&=p(c)\overline\varphi(1)=p(c)\beta\\
&=(p_0+p_1c+\cdots p_nc^n)\left(b_0d^*c^*q_{-1}(c^*)+\sum_{i\geq 0}c^i\left(\sum_{j=0}^ib_jc^*q_{i-j}(c^*)\right) \right)\\
&=p_0b_0d^*c^*q_{-1}(c^*)+p_0b_0c^*q_0(c^*)+c\Big(p_1b_0c^*q_0(c^*)+p_0\big(b_0c^*q_1(c^*)+b_1c^*q_0(c^*)\big)\Big)+\\
+c^2\Big(p_2&b_0c^*q_0(c^*)+p_1\big(b_0c^*q_1(c^*)+b_1c^*q_0(c^*)\big)+p_0\big(b_0c^*q_2(c^*)+b_1c^*q_1(c^*)+b_2c^*q_0(c^*)\big)\Big)+\cdots\\
&=d^*c^*q_{-1}(c^*)+c^*q_0(c^*)+cc^*q_1(c^*)+c^2c^*q_2(c^*)+\cdots=\varphi(p(c)).
\end{align*}
So injectivity with respect to left ideals of the form $Rp(c)$, for $p(x)\in \mathcal P$, has been shown.

On the other hand, for (1),
%
%
let now $\psi:J\to d^*c^*K[c^*]+[[c]](c^*K[c^*])$ be a non-zero homomorphism of  $R$-modules. Since $\psi(w)=\psi(w\cdot w)=w\psi(w)$, and $\psi(d^*(c^*)^i)=\psi(w\cdot d^*(c^*)^i)=w\psi(d^*(c^*)^i)$ for each $i\geq 0$, it follows that 
\[\psi(w)=d^*c^*q'_{-1}(c^*), \text{ and }\psi(d^*(c^*)^i)=d^*c^*q'_{i}(c^*)\]
for suitable polynomials $ q'_i(x)\in K[x]$, $i\geq -1$. Setting
\[\overline\psi(1)=d^*c^*q'_{-1}(c^*)+\sum_{j\geq 0}c^jdd^*c^*q'_j(c^*),\]
we extend $\psi$ to an homomorphism $\overline\psi:R\to d^*c^*K[c^*]+[[c]](c^*K[c^*])$. Indeed, $\overline\psi(w)=w\overline\psi(1)=d^*c^*q'_{-1}(c^*)$, and
$\overline\psi(d^*(c^*)^i)=d^*(c^*)^i\overline\psi(1)=d^*c^*q'_{i}(c^*)$ for all $ i\geq 0$.
\end{proof}

As we noticed in Lemma~\ref{lemma:imaggini}, in order to determine all the indecomposable injective $R$-modules, we have to study the quotients of $E(R)$, and in particular of $E(Rc^*)$.  

The injective $R$-module $E(Rc^*)=d^*c^*K[c^*]+[[c]](c^*K[c^*])$ has the following  $R$-submodule:
\[E(Rc^*)^b:=\{d^*c^*q_{-1}(c^*)+\sum_{i\geq 0}c^ic^*q_i(c^*)\mid q_i(x)\in K[x], \ \max_i\deg q_i(x)\in\mathbb N\}.\]
The  $R$-module $E(Rc^*)^b$ consists of the elements in $E(Rc^*)$ for which the polynomials $\{q_i(c^*):i\geq -1\}$ in $c^*$ have bounded degree. It is easy to check that 
\[\bigoplus_{j\geq 1}E\big(Rd^*(c^*)^j\big)=\big(Kd^*c^*+[[c]](K(v-cc^*)c^*)\big)[c^*]\leq E(Rc^*)^b\leq E(Rc^*).\]

\begin{proposition}
The submodule $\bigoplus_{j\geq 1}Rd^*(c^*)^j$ is essential in $E(Rc^*)$ and therefore
\[E\Big(\bigoplus_{j\geq 1}Rd^*(c^*)^j\Big)=E\Big(\bigoplus_{j\geq 1}E\big(Rd^*(c^*)^j\big)\Big)=E\big(E(Rc^*)^b\big)=E(Rc^*).\]
\end{proposition}
\begin{proof}
Let $d^*c^*q_{-1}(c^*)+\displaystyle\sum_{i\geq 0}c^ic^*q_i(c^*)$ be a non zero element in $E(Rc^*)$. Denote by $\ell$ the minimum $i$ such that $q_i(c^*)\not=0$. If $\ell=-1$, then
\[w\big(d^*c^*q_{-1}(c^*)+\displaystyle\sum_{i\geq 0}c^ic^*q_i(c^*)\big)=d^*c^*q_{-1}(c^*)\in \bigoplus_{j\geq 1}Rd^*(c^*)^i.\]
If $\ell\geq 0$, then
\[d^*(c^*)^\ell\big(d^*c^*q_{-1}(c^*)+\displaystyle\sum_{i\geq 0}c^ic^*q_i(c^*)\big)=
d^*(c^*)^\ell\big(\displaystyle\sum_{i\geq \ell}c^ic^*q_i(c^*)\big)=
d^*c^*q_{\ell}(c^*)\]
belongs to $\bigoplus_{j\geq 1}Rd^*(c^*)^i$.
\end{proof}

\begin{lemma}\label{lemma:Laurentmodules}
It is 
\begin{align*}
J\cdot E(Rc^*)=J\cdot E(Rc^*)^b&=d^*c^*K[c^*]\\
&\leq \bigoplus_{j\geq 1} Kd^*(c^*)^j+[[c]](Kdd^*(c^*)^j)\\
&= \bigoplus_{j\geq 1} E(Rd^*(c^*)^j).\end{align*}
\end{lemma}
\begin{proof}
Recall that $J=Rw\oplus\bigoplus_{i\geq 0}Rd^*(c^*)^i$. Then
\[J\cdot [[c]](c^*K[c^*])=d^*(c^*K[c^*]).\]
Therefore 
\[J\cdot E(Rc^*)=d^*c^*K[c^*]=J\cdot E(Rc^*)^b.\]
The remaining (dis)equalities are clear.
\end{proof}
We have the following short exact sequences of left $R$-modules:
\[0\to\bigoplus_{j\geq 1} E(Rd^*(c^*)^j)\to E(Rc^*)^b\to E(Rc^*)^b/\bigoplus_{j\geq 1} E(Rd^*(c^*)^j)\to 0,\quad\text{and}
\]
\[0\to\bigoplus_{j\geq 1} E(Rc^*)^b\to E(Rc^*)\to E(Rc^*)/E(Rc^*)^b\to 0.
\]
 Also if it is not strictly connected with the aim of this paper, we think it could be interesting to observe the following:
\begin{proposition}\label{prop:LaurentSeries}
The quotient $E(Rc^*)^b/\bigoplus_{j\geq 1} E(Rd^*(c^*)^j)$ is isomorphic as $K[\c,\c^{-1}]$-module to the Laurent formal series $K((\c))$.
\end{proposition}
\begin{proof}
By Lemma~\ref{lemma:Laurentmodules}, $J\cdot \left(E(Rc^*)^b/\bigoplus_{j\geq 1} E(Rd^*(c^*)^j)\right)=0$ and hence the quotient is a $R/J=K[\c,\c^{-1}]$-module.\\
Let $d^*c^*q_{-1}(c^*)+\sum_{i\geq 0}c^ic^*q_i(c^*)$ be an element of $E(Rc^*)^b$ and $\n\in\mathbb N$ the maximum degree of the polynomials $q_i(x)$, $i\geq -1$. Set
\[q_i(x)=a_{0,i}+a_{1,i}x+\cdots+a_{\n,i}x^{\n}\quad i\geq 0.\]
Observe that $d^*c^*q_{-1}(c^*)$ belongs to $\bigoplus_{j\geq 1} E(Rd^*(c^*)^j)$; therefore
\[d^*c^*q_{-1}(c^*)+\sum_{i\geq 0}c^ic^*q_i(c^*)\equiv \left(\sum_{i\geq 0}c^iq_i(c^*)\right)c^*\mod \bigoplus_{j\geq 1} E(Rd^*(c^*)^j).\]
Now, making explicit the polynomials $q_i(c^*)$, we get
\begin{align*}
\left(\sum_{i\geq 0}c^iq_i(c^*)\right)c^*&=\left(\sum_{i\geq 0}c^i(a_{0,i}+a_{1,i}c^*+\cdots+a_{\n,i}(c^*)^{\n})\right)c^*\\
&=\Big(a_{\n,0}(c^*)^{\n}+\left[a_{\n-1,0}(c^*)^{\n-1}+a_{\n,1}c(c^*)^{\n}\right]+\cdots+\\
&+\big[a_{0,0}v+a_{1,1}cc^*+\cdots+a_{\n,\n}c^{\n}(c^*)^{\n}\big]+\\
&+\big[a_{0,1}c+\cdots+a_{\n,\n+1}c^{\n+1}(c^*)^{\n}\big]+\cdots+\\
&+\big[a_{0,\ell}c^\ell +\cdots+a_{\n,\n+\ell}c^{\n+\ell}(c^*)^{\n}\big]+\cdots\Big)c^*
\end{align*}
Given any non-negative integer $m$, for $h<m$ we have
\[c^h(c^*)^m =(c^*)^{m-h}-\big(c^{h-1}(v-cc^*)(c^*)^{m-1}+\cdots+(v-cc^*)(c^*)^{m-h}\big);\]
for any $h\geq m$ we have
\[c^h(c^*)^m =c^{h-m}-\big(c^{h-1}(v-cc^*)(c^*)^{m-1}+\cdots+c^{h-m}(v-cc^*)\big).\]
In particular, letting $m=1,\dots,\n$, we have the following congruences modulo $\bigoplus_{j\geq 1} E(Rd^*(c^*)^j)=\big(Kd^*c^*+[[c]]K(v-cc^*)c^*\big)[c^*]$:
\begin{align*}\left(\sum_{i\geq 0}c^iq_i(c^*)\right)&c^*\equiv \Big( a_{\n,0}(c^*)^{\n}+\left[a_{\n-1,0}+a_{\n,1}\right](c^*)^{\n-1}+\cdots+\\
&+\big[a_{0,0}+a_{1,1}+\cdots+a_{\n,\n}\big]v+\big[a_{0,1}+\cdots+a_{\n,\n+1}\big]c+\\
&+\cdots+\big[a_{0,\ell}+\cdots+a_{\n,\n+\ell}\big]c^\ell+\cdots\Big)c^*.
\end{align*}
Moreover the latter belongs to $\bigoplus_{j\geq 1} E(Rd^*(c^*)^j)$ if and only if 
\[0=a_{\n,0}=a_{\n-1,0}+a_{\n,1}=\cdots=a_{0,0}+a_{1,1}+\cdots+a_{\n,\n}\text{ and}
\]
\[0=a_{0,\ell}+a_{1,1+\ell}+\cdots+a_{\n,\n+\ell}\quad\forall \ell\geq 1.\]
Therefore a irredundant  set of representatives for the elements of the quotient $E(Rc^*)^b/\bigoplus_{j\geq 1} E(Rd^*(c^*)^j)$ is given by the elements of $E(Rc^*)^b$ of the form
\[\left(a'_{-m}(c^*)^m+\cdots+a'_{-1}c^*+a'_0v+\sum_{i\geq 1}a'_ic^i\right)c^*\quad m\in\mathbb N, a'_j\in K\ \forall j\geq -m.\]
This representatives have a nice behaviour with respect to the $K[\c,\c^{-1}]$-module structure of the quotient: indeed we have
\begin{align*}
\c\cdot \Big(a'_{-m}(c^*)^m&+\cdots+a'_{-1}c^*+a'_0v+\sum_{i\geq 1}a'_ic^i\Big)c^*\equiv\\
&\equiv \Big(a'_{-m}(c^*)^{m-1}+\cdots+a'_{-1}v+a'_0c+\sum_{i\geq 1}a'_ic^{i+1}\Big)c^*,\quad\text{and}\end{align*}
\begin{align*}
\c^{-1}\cdot \Big(a'_{-m}(c^*)^m&+\cdots+a'_{-1}c^*+a'_0v+\sum_{i\geq 1}a'_ic^i\Big)c^*\equiv\\
&\equiv \Big(a'_{-m}(c^*)^{m+1}+\cdots+a'_{-1}(c^*)^2+a'_0c^*+\sum_{i\geq 1}a'_ic^{i-1}\Big)c^*.\end{align*}

Setting  
\[\left(a'_{-m}(c^*)^m+\cdots+a'_{-1}c^*+a'_0v+\sum_{i\geq 1}a'_ic^i\right)c^*+\bigoplus_{j\geq 1} E(Rd^*(c^*)^j)\mapsto \sum_{i\geq -m}a'_i\c^i\]
for each $m\geq 0$ and each sequence of scalars $(a'_i)_{i\geq -m}$ in $K$, one therefore 
defines an isomorphism of $K[\c,\c^{-1}]$-modules between the quotient $E(Rc^*)^b/\bigoplus_{j\geq 1} E(Rd^*(c^*)^j)$ and $K((\c))$.\end{proof}

We are now ready to prove that

\begin{proposition}
In the $K[\c,\c^{-1}]$-module $E(Rc^*)/\bigoplus_{j\geq 1} E(Rd^*(c^*)^j)$ there are no $p(\c)$-torsion elements for any polynomial $p(x)\in K[x]$  with $p(0)=1$.
\end{proposition}
\begin{proof}
We will prove that there are no $p(\c)$ torsion elements in both the quotients $E(Rc^*)/E(Rc^*)^b$ and $E(Rc^*)^b/\bigoplus_{j\geq 1} E(Rd^*(c^*)^j)$. The first is easy to check: indeed for any element $d^*c^*q_{-1}(c^*)+\sum_{i\geq 0}c^ic^*q_i(c^*)$ in $E(Rc^*)$ the product
\[p(c)\cdot\big(d^*c^*q_{-1}(c^*)+\sum_{i\geq 0}c^ic^*q_i(c^*)\big)\]
belongs to $E(Rc^*)^b$ if and only if all the degrees of the polynomials $q_i(x)$, $i\geq 0$, is bounded, and because $K[\c,\c^{-1}]$ is a domain, this happens if and only if $d^*c^*q_{-1}(c^*)+\sum_{i\geq 0}c^ic^*q_i(c^*)$ was already in $E(Rc^*)^b$. For the second, we proved in Proposition~\ref{prop:LaurentSeries} that $E(Rc^*)^b/\bigoplus_{j\geq 1} E(Rd^*(c^*)^j)$ is $K[\c,\c^{-1}]$-isomorphic to the Laurent formal series $K((\c))$, and the latter has no $p(\c)$-torsion elements.
\end{proof}

\begin{corollary}\label{cor:finale}
The left $R$-module $E(Rc^*)/\bigoplus_{j\geq 1} E(Rd^*(c^*)^j)$ is $K[\c,\c^{-1}]$-isomorphic to a direct sum of copies of $K(\c)$ (see Remark~\ref{rem:K(C)}).
\end{corollary}
\begin{proof}
Being  $R$ hereditary, and by Lemma~\ref{lemma:Laurentmodules}, 
we have that the quotient $E(Rc^*)/\bigoplus_{j\geq 1} E(Rd^*(c^*)^j)$  is  an  injective $R/J=K[\c,\c^{-1}]$-module. Thus it is a  direct sum of indecomposable injective $K[\c,\c^{-1}]$-modules, and so of modules isomorphic to either $K(\c)$ or to $\varinjlim_i K[\c,\c^{-1}]/K[\c,\c^{-1}]f(\c)^i$ for  suitable irreducible polynomials $f(x)\in K[x]$ with $f(0)=1$ (see \cite[Theorem 3.48]{Lam1}). Then we conclude by the previous proposition.
\end{proof}

Now we have all the results in place to establish our main result, the description of all indecomposable injective left $R$-modules.

\begin{theorem}\label{thm:main}
Let $R$ denote the Leavitt path algebra
$L_K (\mathcal T )$. 
The complete list of indecomposable injective left $R$-modules is
$E(Rw)$, $K(\c)$, and the Pr\"ufer modules $U^f$, for each irreducible polynomial $f(x)$ in $K[x]$ with $f(0)=1$.
\end{theorem}
\begin{proof}
Let $I$ be an injective indecomposable left $R$-module. By Lemma~\ref{lemma:imaggini} there exists an epimorphism
\[\xymatrix{E({}_RR)=E(Rc^*)\oplus E(Rd^*)\oplus E(Rw)\ar[rr]^-\varphi&& I.}\]
 
 If $\varphi(E(Rw))$ (or $\varphi(E(Rd^*))$) is not 0, then it is an injective  $R$-submodule of $I$ and hence, being $I$ indecomposable, we get $\varphi(E(Rw))=I$ (or $\varphi(E(Rd^*))=I$): in both the cases $I$ is an homomorphic image of $E(Rw)\cong E(Rd^*)$ (Case 1). 

If, on the other hand, $\varphi(E(Rw))=0=\varphi(E(Rd^*))$, then $\varphi(E(Rc^*))=I$.
%
%
If the restriction of $\varphi$ to $\bigoplus_{j\geq 1} E(Rd^*(c^*)^j)$ is not zero, then there exists an epimorphism  \[\xymatrix{E(Rw)\cong E(Rd^*(c^*)^{\overline j})\ar@{->>}[r]&I}\] where $\overline j$ is a suitable integer $\geq 1$ and we return to the previous Case 1. Otherwise, $\varphi$ induces an epimorphism 
\[\xymatrix{E(Rc^*)/\bigoplus_{j\geq 1} E(Rd^*(c^*)^j)\ar@{->>}[rr]&&I}.\]
In such a case, by Corollary~\ref{cor:finale}, we conclude that $I$ is an homomorphic image of the  $R$-module $K(\c)$ (Case 2).

\noindent Let us study the indecomposable images of $E(Rw)$ (Case 1) and of $K(\c)$ (Case 2).

 Case 1. $\xymatrix{E(Rw)\ar@{->>}[r]^-\varphi &I}$. If $\varphi$ is injective, then $I$ is isomorphic to  $E(Rw)$. Assume $\Ker\varphi\not=0$.
Since $Rw$ is simple and essential in $E(Rw)$, it is contained in $\Ker\varphi$. It is easy to check that
\[Rw=(Rw\oplus\bigoplus_{i\geq 0}Rd^*(c^*)^i)\cdot\{k_{-1}w+\sum_{i\geq 0}k_ic^id\mid k_i\in K\}=J\cdot E(Rw).\]
Thus $I$ is an indecomposable injective left $R$-module such that $J\cdot I=0$. Therefore $I$ is an indecomposable injective $R/J=K[\c,\c^{-1}]$-module and hence it is isomorphic to either a Pr\"ufer module associated to an irreducible polynomial $f(x)\in K[x]$ with $f(0)=1$, or to $K(\c)$. 


Case 2. $\xymatrix{K(\c)\ar@{->>}[r]^-\varphi &I}$.  If $\varphi $ is injective, then $I$ is isomorphic to  $K(\c)$. Otherwise  $I$ has $K[\c,\c^{-1}]$-torsion elements and hence it is isomorphic to a Pr\"ufer module $U^f$ associated to an irreducible polynomial $f(x)\in K[x]$ with $f(0)=1$. 
%
%
%
\end{proof}


We conclude the paper restating Theorem~\ref{thm:main} in terms of modules which are explicitly described as modules over the Jacobson algebra $K\langle X,Y\mid XY=1\rangle$. The socle of the Jacobson algebra is the two-sided ideal generated by $1-YX$; setting $\overline X:=X+\langle 1-YX\rangle$ in the quotient $K\langle X,Y\mid XY=1\rangle/\langle 1-YX\rangle$, we have $Y+\langle 1-YX\rangle={\overline X}^{-1}$ and $K\langle X,Y\mid XY=1\rangle/\langle 1-YX\rangle$ is equal to the ring of Laurent polynomials $K[\overline X, {\overline X}^{-1}]$.

\begin{corollary}
The complete list of indecomposable injective left $K\langle X,Y\mid XY=1\rangle$-modules is:
\[K[[Y]](1-YX),\quad K(\overline X),\quad  \text{and}\quad \varinjlim K[\overline X, {\overline X}^{-1}]/f^n(\overline X)K[\overline X, {\overline X}^{-1}]\]
for each irreducible polynomial $f(x)$ in $K[x]$ with $f(0)=1$.
\end{corollary}
\begin{proof}
We get the description of the indecomposable injective modules over the Jacobson algebra applying the isomorphism between $L_K(\mathcal T)$ and $K\langle X,Y\mid XY=1\rangle$ described at page~\pageref{isoJac}. Only how the first injective indecomposable module was obtained requires some further detail. Since $w\mapsto 1-YX$, $c\mapsto Y^2X$, and $d\mapsto(Y-Y^2X)$ the left $L_K(\mathcal T)$-module $E(Rw)=Kw+[[c]](Kd)$ corresponds to
$K(1-YX)+[[Y^2X]](K(Y-Y^2X))$. A generic element in $K(1-YX)+[[Y^2X]](K(Y-Y^2X))$ has the form
\[k_{-1}(1-YX)+k_0(Y-Y^2X)+k_1 Y^2X(Y-Y^2X)+\cdots+k_i (Y^2X)^i(Y-Y^2X)+\cdots\]
For each $i\geq 0$, we have
\[(Y^2X)^i(Y-Y^2X)=\underbrace{Y^2X\cdots Y^2X}_{i\text{ times}}Y(1-YX)=Y^{i+1}(1-YX).\]
Therefore $K(1-YX)+[[Y^2X]](K(Y-Y^2X))=K[[Y]](1-YX)$.
\end{proof}

\textbf{Acknowledgement} The authors are deeply grateful to the anonymous referee for the care with which our work was reviewed and for her/his valuable suggestions.

\end{document}